\documentclass[a4paper,11pt,twoside]{amsart}
\usepackage[english]{babel}
\usepackage[utf8]{inputenc}

\usepackage[a4paper,inner=3cm,outer=3cm,top=4cm,bottom=4cm]{geometry}
\usepackage{fancyhdr}
\usepackage{bold-extra}
\usepackage{ mathrsfs }

\usepackage{comment}
\usepackage{graphics}
\usepackage{aliascnt}
\usepackage[hidelinks]{hyperref}

\usepackage{amsmath}
\usepackage{amsfonts}
\usepackage{amssymb}
\usepackage{amsthm}
\usepackage{comment}
\usepackage{mathtools}
\usepackage{enumerate}
\usepackage{enumitem} 

\newtheorem{theorem}{Theorem}[section]
\newtheorem*{theorem*}{Theorem}

\newtheorem*{corollary*}{Corollary}

\newtheorem{lemma}[theorem]{Lemma}
\newtheorem{proposition}[theorem]{Proposition}
\newtheorem*{proposition*}{Proposition}
\theoremstyle{definition}
\newtheorem{definition}[theorem]{Definition}
\newtheorem{remark}[theorem]{Remark}
\newtheorem*{remark*}{Remark}

\numberwithin{equation}{section}

\newcommand\eps{\varepsilon}

\newcommand\R{\mathbb{R}}
\newcommand\Z{\mathbb{Z}}

\newcommand\C{\mathbb{C}}

\newcommand{\wt}{\widetilde}
\newcommand{\ve}{\mathbf}
\newcommand{\Mod}[1]{\ (\mathrm{mod}\ #1)}

\newcommand{\PS}{\operatorname{PS}}
\newcommand{\Poly}{\operatorname{Poly}}

\begin{document}
\title{Linear equations in Piatetski-Shapiro primes}

\author[Shao]{Xuancheng Shao}
\address{Department of Mathematics, University of Kentucky\\
Lexington, KY 40506\\
USA}
\email{xuancheng.shao@uky.edu}

\author{Yu-Chen Sun}
\address{
Department of Mathematics,
University of Bristol, Woodland Rd, Bristol BS8 1UG}
\email{yuchensun93@163.com}

\begin{abstract}
We establish discorrelation estimates between the Piatetski-Shapiro prime set
\[
\mathcal{P}_{\gamma} := \{p \text{ is prime and } p = \lfloor n^{1/\gamma} \rfloor \text{ for some } n \in \mathbb{N}\}
\]
and arbitrary nilsequences when $\gamma \in (0,1)$ is sufficiently close to $1$. This extends earlier works which treated linear or polynomial exponential phase functions. As an application, we establish an asymptotic formula for the number of solutions in $\mathcal{P}_{\gamma}$ to any ``finite-complexity" system of linear equations, including for the number of  $k$-term arithmetic progressions in  $\mathcal{P}_{\gamma}$ up to a threshold $N$ for any given $k \geq 3$. Furthermore, we show that there exists an absolute constant $C>0$ such that if
\[
1 - 2^{-Ck} < \gamma < 1,
\]
then the Piatetski-Shapiro primes $\mathcal{P}_{\gamma}$ contain infinitely many non-trivial $k$-term arithmetic progressions. This significantly improves upon the previous range of $\gamma$ obtained by Li and Pan \cite{LiPan}, which is of triple exponential type.
\end{abstract}

\maketitle

\section{Introduction}

The celebrated Green--Tao theorem \cite{GT-kAP} asserts that the primes contain arbitrarily long arithmetic progressions. In fact, their theorem gives a relative density statement. If $\mathbb{P}$ denotes the set of primes and $A\subseteq\mathbb{P}$ satisfies
\[
\limsup_{N\to\infty}\frac{|A\cap[1,N]|}{\pi(N)}>0,
\qquad \pi(N):=|\mathbb{P}\cap[1,N]|,
\]
then $A$ contains non-trivial arithmetic progressions of every finite length.

More generally, Green and Tao \cite{GT-linear,GT-nil}, together with Green, Tao and Ziegler \cite{GTZ}, established the expected asymptotic formulae for the number of $k$-term arithmetic progressions in the primes, and more generally for any finite-complexity system of affine-linear forms in the primes. For further related work on linear or polynomial patterns in sparse sets, see \cite{BST,CFZ,CookMagyarTitichetrakun,MatthiesenWang,TaoZiegler-poly,TaoZiegler-multidim}.

The problem of obtaining good quantitative bounds in finding $k$-term arithmetic progressions in subsets of primes has received considerable interest. For $k=3$, the quasi-polynomial bound from the work of Kelley and Meka \cite{KelleyMeka} for subsets of integers implies the same type of bound for subsets of primes. For fixed $k \geq 4$, the strongest bounds currently known are due to Ter\"{a}v\"{a}inen and Wang \cite{TeravainenWang2026}, improving earlier work of Rimani\'{c} and Wolf \cite{RimanicWolf2019}. More precisely, if $A\subseteq\mathbb{P}\cap[1,N]$ contains no non-trivial $k$-term arithmetic progression, then
\[
\frac{|A|}{\pi(N)}\ll_k
\begin{cases}
(\log\log N)^{-c_4}, & k=4,\\[2pt]
\exp\!\bigl(-(\log\log\log N)^{c_k}\bigr), & k\geq5,
\end{cases}
\]
where $c_k>0$ depends only on $k$.

This naturally raises the question of whether arithmetic progressions and other additive structures can also be found in much sparser, arithmetically defined subsets of the positive integers. A natural family in which to investigate this question is given by the Piatetski-Shapiro sequences, whose arithmetic properties form an important and extensively studied topic in analytic number theory. For a fixed real number $c>1$, define
$$
\PS_c = \{ \lfloor n^c\rfloor: n \in \Z_{>0}\}.
$$
The elements of $\PS_c$ are called Piatetski-Shapiro numbers. For fixed $c>1$, we have $|\PS_c\cap[1,x]|=x^{1/c}+O(1)$, so their density in the positive integers decays as a power of $x$. The primes belonging to this sequence are called Piatetski-Shapiro primes; they form the set $\mathbb{P}\cap\PS_c=\mathcal{P}_{1/c}$.

The study of prime values in these sequences goes back to Piatetski-Shapiro \cite{PS}, who proved in 1953 that $\PS_c$ contains infinitely many primes for $1<c<12/11$, and established the asymptotic formula
$$
\pi_c(x) := \#\{p \leq x: p \in \PS_c\} \sim \frac{x^{\frac{1}{c}}}{\log x}.
$$
Subsequent work by various authors \cite{Heath-Brown, Liu-Rivat, Jia, Baker-Harman-Rivat, Kumchev, Rivat-Sargos, Rivat-Wu} progressively enlarged the admissible range for $c$. The current record for the lower bound $\pi_c(x) \gg x^{\frac{1}{c}}/\log x$ is $1 < c < 243/205$ due to Rivat and Wu \cite{Rivat-Wu}. 

Additive problems involving Piatetski-Shapiro primes have also been studied extensively. For example, by adapting the Maynard--Tao sieve, Li and Pan \cite{LiPan-gaps} obtained a Maynard--Tao type result showing that, for every $1<c<9/8$ and every integer $m\geq 1$, there exists $k_0=k_0(c,m)$ such that
$$
\{\lfloor n^c\rfloor,\lfloor(n+1)^c\rfloor,\ldots,\lfloor(n+k_0)^c\rfloor\}
$$
contains at least $m+1$ primes for infinitely many positive integers $n$.

Li and Pan \cite{LiPan} also established a Green--Tao type result for Piatetski-Shapiro primes. For each $k\geq3$, the Piatetski-Shapiro primes contain infinitely many non-trivial $k$-term arithmetic progressions when $c>1$ is sufficiently close to $1$ in terms of $k$. For further work on additive problems involving Piatetski-Shapiro primes, see \cite{BF,AG,Mirek2015,SDP}.

The goal of the present paper is to obtain an asymptotic formula for linear equations in Piatetski-Shapiro primes, generalizing the Green-Tao theorem on linear equations in primes. Our first main result, Theorem~\ref{thm:lin-eq-PS}, establishes the expected asymptotic formula for fixed finite-complexity systems when the Piatetski-Shapiro exponent is sufficiently close to $1$. Let $\Lambda$ be the von Mangoldt function on the positive integers, defined by $\Lambda(n)=\log p$ if $n=p^k$ for a prime $p$ and an integer $k\geq1$, and $\Lambda(n)=0$ otherwise. For $\gamma \in (0,1)$, define $\Lambda_{\gamma}$ to be the normalized von Mangoldt weight for the Piatetski-Shapiro sequence $\PS_{1/\gamma}$ by
$$
\Lambda_{\gamma}(n) = \frac{n^{1-\gamma}}{\gamma} \Lambda(n) \cdot 1_{n \in \PS_{1/\gamma}}.
$$
For $\gamma$ sufficiently close to $1$, this normalization gives $\Lambda_{\gamma}$ asymptotic mean $1$.  {By convention, both $\Lambda$ and $\Lambda_\gamma$ are extended to $\Z$ by setting them equal to $0$ on non-positive integers.}

\begin{theorem}[Linear equations in Piatetski-Shapiro primes]\label{thm:lin-eq-PS}
Let $X \geq 3$. Let $d,t,L\geq 1$. Let $\Psi = (\psi_1,\cdots,\psi_t)$ be a system of nonconstant affine-linear forms, where each $\psi_i:\Z^d\rightarrow\Z$ has the form $\psi_i(\ve{x}) = \dot{\psi_i}\cdot\ve{x} + \psi_i(0)$ with $\dot{\psi_i}\in\Z^d$ and $\psi_i(0) \in \Z$ satisfying $\|\dot{\psi_i}\|_{\infty} \leq L$ and $|\psi_i(0)| \leq LX$. Suppose that $\dot{\psi_i}$ and $\dot{\psi_j}$ are linearly independent whenever $i \neq j$. Let $K \subset [-X, X]^d$ be a convex body. Then
$$
\sum_{\ve{n} \in K \cap \Z^d} \prod_{i=1}^t \Lambda_{\gamma}(\psi_i(\ve{n})) = \sum_{\ve{n} \in K \cap \Z^d} \prod_{i=1}^t \Lambda(\psi_i(\ve{n})) + o_{d,t,L,\gamma}(X^d)
$$
for each fixed $\gamma \in (0,1)$ provided that $1-\gamma$ is sufficiently small in terms of $d,t,L$.
\end{theorem}

Thus the sum on the left-hand side involving $\Lambda_{\gamma}$ has the same asymptotic formula as the sum on the right-hand side involving $\Lambda$, which is given by the Green-Tao theorem \cite{GT-linear}. As immediate corollaries, we can deduce a version of Vinogradov's three primes theorem with Piatetski-Shapiro primes and we can count asymptotically the number of $k$-APs in Piatetski-Shapiro primes in $\PS_c$ when $c$ is sufficiently close to $1$.

\begin{corollary*}[Vinogradov's theorem with Piatetski-Shapiro primes]
Let $N \geq 3$ be an odd positive integer. Then
$$
\sum_{\substack{n_1,n_2,n_3 \\ N=n_1+n_2+n_3}} \Lambda_{\gamma}(n_1)\Lambda_{\gamma}(n_2)\Lambda_{\gamma}(n_3) = \sum_{\substack{n_1,n_2,n_3 \\ N=n_1+n_2+n_3}} \Lambda(n_1)\Lambda(n_2)\Lambda(n_3) + o(N^2)
$$
for each fixed $\gamma \in (0,1)$ provided that $1-\gamma$ is sufficiently small.
\end{corollary*}

\begin{remark*}
In this problem, explicit bounds for how small $1-\gamma$ is required have been obtained in the literature, starting from the work of Balog and Friedlander \cite{BF}. Du, Pan and the second author \cite{SDP} obtained the range $1<c<73/64$,
where $c=1/\gamma$. This was subsequently improved to
$1<c<569/498$ by Lu, Guo and Guo \cite{LuGuoGuo2026}.
\end{remark*}

The following corollary gives an asymptotic formula for the number of $k$-term arithmetic progressions in Piatetski-Shapiro primes.

\begin{corollary*}[$k$-APs in Piatetski-Shapiro primes]
Let $X \geq 3$. For any fixed positive integer $k \geq 3$ we have
$$
\sum_{\substack{1 \leq n,m \leq X \\ n+(k-1)m \leq X}} \prod_{i=0}^{k-1} \Lambda_{\gamma}(n + im) = \sum_{\substack{1 \leq n,m \leq X \\ n+(k-1)m \leq X}} \prod_{i=0}^{k-1} \Lambda(n + im) + o(X^2)
$$
for each fixed $\gamma \in (0,1)$ provided that $1-\gamma$ is sufficiently small in terms of $k$.
\end{corollary*}

This asymptotic result is new for $k \geq 4$. As mentioned above, Li and Pan \cite{LiPan} proved the existence of infinitely many $k$-term arithmetic progressions in Piatetski-Shapiro primes for every $k \geq 3$, provided that $1-\gamma$ is sufficiently small in terms of $k$. The dependence of $1-\gamma$ on $k$ is of triple exponential type:
$$
1-\gamma < 2^{-2^{k^2 \cdot 4^k}}.
$$
We greatly improve the quantitative dependence of $1-\gamma$ on $k$ by proving the following theorem.

\begin{theorem}\label{t:improvement_li-pan}
Let $k \ge 3$. Suppose that
\[
1-2^{-Ck}< \gamma < 1
\]
for some sufficiently large absolute constant $C$. Then there are infinitely many non-trivial $k$-term arithmetic progressions in Piatetski-Shapiro primes in $\PS_{1/\gamma}$.
\end{theorem}

Let us briefly discuss the proof ideas. To prove Theorem~\ref{thm:lin-eq-PS}, we first establish nilsequence discorrelation for $\Lambda_\gamma-\Lambda$ in a range of $\gamma$ independent of the nilmanifold dimension (Proposition~\ref{prop:supp-BF}). A pseudorandom majorant satisfying the linear forms condition (Proposition~\ref{prop:majorant}) allows the relative inverse theorem to convert this discorrelation into Gowers uniformity of $\Lambda_{\gamma}-\Lambda$. The generalized von Neumann theorem then leads to the required asymptotic formula.

For Theorem~\ref{t:improvement_li-pan}, the key to improving on Li and Pan's result is to bound exponential sums of the form
\[
 \left|\sum_{n\in\mathcal I}e(F(n)+\theta n+\theta_0)\right|,
 \qquad F(x)=\sum_{i=1}^s h_i(\alpha_i x+\beta_i)^\gamma,
\]
where $\mathcal I$ is an interval of integers and the affine-linear forms $\alpha_i x+\beta_i$ are positive on $\mathcal I$. Under suitable assumptions on the coefficients, Li and Pan's argument \cite{LiPan} uses van der Corput estimates with derivatives of order up to $2^{O(s\log(2s))}$. We instead approximate the phase on short intervals by polynomials of degree $100s$. A Vandermonde argument gives a suitable coefficient of degree between $2$ and $s+1$, and the polynomial exponential sum estimate of \cite[Theorem 5]{B17} then gives a saving whose exponent is a fixed negative power of $s$. This leads to the improved range in Theorem~\ref{t:improvement_li-pan}. The details are given in Section~\ref{sec:pseudo}.

The rest of the paper is organized as follows. Section~\ref{sec:prelim} records the notation and preliminary results. The nilsequence discorrelation estimate is proved in Section~\ref{sec:nil-orthogonality}. Section~\ref{sec:pseudo} constructs the pseudorandom majorant for the Piatetski-Shapiro primes. Section~\ref{sec:deduction} establishes Gowers uniformity and deduces both Theorem~\ref{thm:lin-eq-PS} and Theorem~\ref{t:improvement_li-pan}.

\subsection*{Acknowledgements}

The authors are grateful to the organizers of the summer school on analytic number theory held at Xi'an Jiaotong University in June 2025, where this work started. XS was supported by NSF grant DMS-2452462.

\subsection*{AI assistance}

ChatGPT (OpenAI) was used by the authors to assist with polishing and proofreading. The authors take full responsibility for the mathematical statements, arguments  and their presentation.

\section{Preliminaries}\label{sec:prelim}

Throughout, $n\sim N$ means $N<n\leq2N$, and $e(x)=e^{2\pi i x}$. We always assume that $1-\gamma$ is sufficiently small and allow all implied constants to depend on $\gamma$.

\subsection{Harmonic analysis results}

\begin{lemma}[Approximation of the sawtooth function]\label{approx-sawtooth}
Let $\psi(x) = \{x\} - 1/2$ be the sawtooth function, and let $H$ be a positive integer. There exists a trigonometric polynomial
$$
\psi^*(x) = \sum_{1 \leq |h| \leq H} a_h e(hx),
$$
where $|a_h| \ll |h|^{-1}$, such that for any real $x$ we have
$$
|\psi(x) - \psi^*(x)| \leq \sum_{|h| < H} b_he(hx),
$$
where $|b_h| \ll H^{-1}$.
\end{lemma}

\begin{proof}
See \cite[Lemma 2.1]{AG}.
\end{proof}

\begin{lemma}[Van der Corput]\label{van-der-corput}
Let $X \geq Y \geq 1$. Let $k \geq 2$ be fixed. Let $f: [X, X+Y] \rightarrow\R$ be a function such that $f^{(k)}(x)$ is continuous. Suppose that $|f^{(k)}(x)| \asymp \Delta$ for some $\Delta > 0$ and all $x \in [X, X+Y]$. Then
$$
\sum_{X < n \leq X+Y} e(f(n)) \ll Y\Delta^{1/(2^k-2)} + Y^{1-2^{2-k}}\Delta^{-1/(2^k-2)}.
$$
\end{lemma}

\begin{proof}
See \cite[Theorem 3]{Robert2016}. For the $k=2$ case, see also \cite[Theorem 2.2]{GKbook}. We will only need the cases $k=2$ and $k=3$.
\end{proof}

\subsection{Nilsequences}

In this subsection we recall the definitions of nilmanifolds and polynomial sequences and the associated notation from \cite{GT12}.

\begin{definition}[Filtered nilmanifold]
Let $d,D \geq 1$ and let $0 < \delta < 1$. A \emph{filtered nilmanifold} $G/\Gamma$ of degree at most $d$, dimension $D$, and complexity at most $1/\delta$ consists of the following data:
\begin{itemize}
\item A connected and simply connected nilpotent Lie group $G$ of dimension $D$;
\item A filtration $G_{\bullet} = (G_i)_{i=0}^{\infty}$ of degree at most $d$, with $G_0=G_1=G$ and all $G_i$ closed connected subgroups of $G$;
\item A lattice (i.e. a discrete and cocompact subgroup) $\Gamma$ of $G$, with $\Gamma_i := \Gamma \cap G_i$ a lattice of $G_i$ for all $i$;
\item A Mal'cev basis for $G/\Gamma$ adapted to $G_{\bullet}$ which is $1/\delta$-rational (see Definitions 2.1 and 2.4 in \cite{GT12}).
\end{itemize}
\end{definition}

The Lipschitz norm of a function $F: G/\Gamma\rightarrow \C$ is defined to be the quantity
$$
\sup_{x \in G/\Gamma}|F(x)| + \sup_{\substack{x, y \in G/\Gamma \\ x \neq y}} \frac{|F(x)-F(y)|}{d_{G/\Gamma}(x,y)},
$$
where $d_{G/\Gamma}$ is a metric on $G/\Gamma$ induced from a metric $d_G$ on $G$ defined using the data associated to the nilmanifold $G/\Gamma$, as in \cite[Definition 2.2]{GT12}. In particular, if $F$ has Lipschitz norm at most $1/\delta$, then $\|F\|_{\infty} \leq 1/\delta$.

\begin{definition}[Polynomial sequences]
Given a filtered nilmanifold $G/\Gamma$ equipped with a filtration $G_{\bullet}$ of degree at most $d$, a \emph{polynomial sequence} $g: \Z\rightarrow G$ (adapted to $G_{\bullet}$) is a map of the form 
$$
g(n) = g_0 g_1^{\binom{n}{1}} \cdots g_d^{\binom{n}{d}},
$$
where $g_i \in G_i$ for every $0 \leq i \leq d$. The collection of all such polynomial sequences is denoted by $\Poly(\Z\rightarrow G)$.
\end{definition}

Given the data described above, the sequence $n\mapsto F(g(n)\Gamma)$ is called a nilsequence. We refer the readers to \cite{GT12} for examples and basic properties of nilsequences.

\subsection{Gowers norms on intervals and cyclic groups}
For an interval $I\subset\Z$, $s\geq1$, and a function $f: I\rightarrow\C$, the Gowers $U^s$-norm of $f$ is defined by the formula
$$
\|f\|_{U^s(I)}^{2^s} = \frac{\sum_{x,h_1,\cdots,h_s\in\Z} \prod_{\omega \in \{0,1\}^s}\mathcal C^{|\omega|}f(x+\omega\cdot\mathbf h)}{\sum_{x,h_1,\cdots,h_s\in\Z} \prod_{\omega \in \{0,1\}^s}\mathcal C^{|\omega|}1_I(x+\omega\cdot\mathbf h)},
$$
where $\mathcal C$ denotes complex conjugation, $\mathbf h = (h_1,\cdots,h_s)$, and $f(x+\omega\cdot\mathbf h)$ is understood to be $0$ when $x+\omega\cdot\mathbf h\notin I$. Thus
$\|1\|_{U^s(I)}=1$ and $\|f\|_{U^1(I)}=|\mathbb E_I f|$. See also \cite[equation (1.17)]{MSTT} and \cite[equation (B.13)]{GT-linear}.

For a finite abelian group $G$, $s \geq 1$, and a function $f: G\rightarrow\C$, the Gowers $U^s$-norm of $f$ is defined by the formula
$$
\|f\|_{U^s(G)}^{2^s} = \frac{1}{|G|^{s+1}}\sum_{x,h_1,\cdots,h_s\in G} \prod_{\omega \in \{0,1\}^s}\mathcal C^{|\omega|}f(x+\omega\cdot\mathbf h).
$$

The Gowers norms on intervals and on cyclic groups are closely related. For $f: I\rightarrow\C$, we can embed $I$ in $\Z/N\Z$ for some $N > 4|I|$ in the natural way, and write $f_0: \Z/N\Z\rightarrow \C$ for the extension by zero of $f$ from $I$ to $\Z/N\Z$. Then we have
$$
 \|f_0\|_{U^s(\Z/N\Z)} =\|1_I\|_{U^s(\Z/N\Z)}\cdot\|f\|_{U^s(I)}.
$$
See \cite[equation (B.14) and Lemma B.5]{GT-linear}.

\section{Nilsequence orthogonality}\label{sec:nil-orthogonality}

In this section, we establish a discorrelation estimate for $\Lambda-\Lambda_{\gamma}$ against nilsequences. Recall that $\Lambda_{\gamma}$ is defined by
$$
\Lambda_{\gamma}(n) = \frac{n^{1-\gamma}}{\gamma}\Lambda(n) \cdot 1_{n \in \PS_{1/\gamma}}.
$$

\begin{proposition}[Discorrelation with nilsequences]\label{prop:supp-BF}
Let $G/\Gamma$ be a filtered nilmanifold of some degree $d$ and dimension $D$, and complexity at most $1/\delta$ for some $\delta \in (0,1)$, and let $F: G/\Gamma\rightarrow \C$ be a Lipschitz function of norm at most $1/\delta$. Let $\gamma \in (0,1)$ be real such that $1-\gamma <1/(4000d\cdot 2^d)$. Then for any polynomial sequence $g \in \Poly(\Z\rightarrow G)$ we have
$$
 \left|\sum_{n\sim N}(\Lambda_\gamma(n)-\Lambda(n))
 F(g(n)\Gamma)\right| \ll_{d,D,\gamma} \delta^{-O_{d,D}(1)} N^{1-c_{d,D}},
$$
where $c_{d,D} > 0$ is a constant depending only on $d,D$.
\end{proposition}

In our application to Theorem \ref{thm:lin-eq-PS}, it is crucial here that the smallness of $1-\gamma$ depends only on the degree $d$ and not on the other parameters of $G/\Gamma$. In the proof of Proposition \ref{prop:supp-BF}, we will first reduce by smooth approximation to the case when $F\in C^{\infty}(G/\Gamma)$, whose Sobolev norm $\|F\|_{W^{m}}$ is controlled for $m = 2^{d+1}D$. See \cite[Definition A.1]{Matthiesen2018} for the precise definitions of these Sobolev norms (which also depend on the Mal'cev basis included in the data for the nilmanifold).

Via a standard Fourier approximation argument carried out in Section \ref{sec:nil-orthogonality1}, Proposition \ref{prop:supp-BF} is then reduced to the following estimate for primes twisted by fractional power phases and nilsequences.

\begin{proposition}[Twists by fractional power phases]\label{prop:supp-smooth-prime-twists}
Let $G/\Gamma$ be a filtered nilmanifold of some degree $d$ and dimension $D$, and complexity at most $1/\delta$ for some $\delta \in (0,1)$, and let $F \in C^{\infty}(G/\Gamma)$ be a smooth function such that $\|F\|_{W^m} \leq 1/\delta$ for $m = 2^{d+1}D$. Let $\gamma \in (3/4,1)$ be real. Let $J \subset (N, 2N]$ be an interval and let $1 \leq |h| \leq N^{1/6}$. Then for any polynomial sequence $g \in \Poly(\Z\rightarrow G)$ we have
$$
 \left|\sum_{n\in J}\Lambda(n)e(hn^\gamma)F(g(n)\Gamma)\right|
 \ll_{d,D,\gamma} \delta^{-O_{d,D}(1)} N^{1-1/(1000d\cdot 2^d)}.
$$
\end{proposition}
For a nonempty integer interval $I$ and a function $u:I\to\C$,
the dual Gowers norm is defined by
\[
 \|u\|_{U^{d+1}(I)^*}
 =\sup_{\substack{f:I\to\C\\\|f\|_{U^{d+1}(I)}\leq1}}
   \left|\frac1{|I|}\sum_{n\in I}f(n)\overline{u(n)}\right|.
\]
Since this norm is invariant under complex conjugation, we have
the duality inequality
\begin{equation}\label{e:dua_ineq}
 \left|\frac1{|I|}\sum_{n\in I}f(n)u(n)\right|
 \leq\|f\|_{U^{d+1}(I)}\|u\|_{U^{d+1}(I)^*}.   
\end{equation}
It is illustrative to consider the (much easier) version of Proposition \ref{prop:supp-smooth-prime-twists}, where $\Lambda(n)$ in the sum is replaced by the constant $1$. In this case, the discorrelation between the function $n\mapsto e(hn^\gamma)$ and nilsequences is qualitatively equivalent to it having small $U^{d+1}(J)$-norm. Indeed we establish such a result for the fractional power phase function in Section \ref{sec:nil-orthogality2}. This is also where the smoothness of $F$ becomes convenient, as it allows us to control the dual $U^{d+1}(J)^*$-norm of the nilsequence. Finally, the proof of Proposition \ref{prop:supp-smooth-prime-twists} with $\Lambda(n)$ is completed using Section \ref{sec:nil-orthogonality3} by reducing to type I and type II estimates.

\subsection{Deduction of Proposition \ref{prop:supp-BF} from Proposition \ref{prop:supp-smooth-prime-twists}}\label{sec:nil-orthogonality1}

We may assume that $\|F\|_{\infty} \leq 1$. It suffices to prove Proposition \ref{prop:supp-BF} for $F \in C^{\infty}(G/\Gamma)$ satisfying $\|F\|_{W^m} \leq 1/\delta$ for $m = 2^{d+1}D$. Once this is proved, we may deduce the general case for Lipschitz $F$ from it as follows. Let $\eps = N^{-c}$ for some constant $c>0$ sufficiently small in terms of $d,D$. Approximate $F$ by a smooth function $F_{\eps}$ (see \cite[Lemma A.3]{Matthiesen2018}) such that
$$
\|F - F_{\eps}\|_{\infty} \ll \eps/\delta, \qquad \|F_{\eps}\|_{W^m} \ll_{d,D} \delta^{-O_{d,D}(1)}\eps^{-2m}
$$
for $m = 2^{d+1}D$.  Thus we may replace $F$ by $F_{\eps}$ in the sum at the cost of an error bounded by
$$
\begin{aligned}
& \Big|\sum_{n \sim N}(\Lambda_{\gamma}(n)-\Lambda(n))(F(g(n)\Gamma) - F_{\eps}(g(n)\Gamma))\Big| \ll 
 \delta^{-1}\eps\sum_{n\sim N}(\Lambda_\gamma(n)+\Lambda(n)) \\
 &\qquad \leq \delta^{-1}\eps \Big(\Big|\sum_{n \sim N} (\Lambda_{\gamma}(n)-\Lambda(n))\Big| + 2\sum_{n\sim N}\Lambda(n)\Big). 
\end{aligned}
$$
Since Proposition \ref{prop:supp-BF} is assumed to be already established for the constant function $1$, the right-hand side above is
$$
\ll_{d,D,\gamma} \delta^{-1}\eps(\delta^{-O_{d,D}(1)} N^{1-c_{d,D}} + N),
$$
which is acceptable by our choice of $\eps$. On the other hand, since Proposition \ref{prop:supp-BF} is assumed to be already established for $F_{\eps}$ (with $\delta$ replaced by $(\delta\eps)^{O_{d,D}(1)}$), we have
$$
 \left|\sum_{n\sim N}(\Lambda_\gamma(n)-\Lambda(n))
 F_{\eps}(g(n)\Gamma)\right| \ll_{d,D,\gamma} (\delta\eps)^{-O_{d,D}(1)} N^{1-c_{d,D}},
$$
which leads to the desired bound with our choice of $\eps$. Henceforth we assume that $F \in C^{\infty}(G/\Gamma)$ and  $\|F\|_{W^m} \leq 1/\delta$.

Let $\theta_d = 1/(4000d \cdot 2^d)$, and thus $1-\gamma < \theta_d$. For $\omega_n = \Lambda(n)F(g(n)\Gamma)$, our goal is to prove that
$$
\Big|\sum_{n\sim N} \frac{n^{1-\gamma}}{\gamma} 1_{n \in \PS_{1/\gamma}} \omega_n - \sum_{n \leq N} \omega_n\Big|  \ll_{d,D} \delta^{-O_{d,D}(1)} N^{1-c_{d,D}}.
$$
Using the identity
\begin{equation}\label{eq:PS-indicator}
1_{n \in \PS_{1/\gamma}} = \lfloor -n^{\gamma}\rfloor - \lfloor -(n+1)^{\gamma}\rfloor = \gamma n^{\gamma-1} + O(n^{\gamma-2}) + \psi(-(n+1)^{\gamma}) - \psi(-n^{\gamma}),
\end{equation}
where $\psi(x) = \{x\}-1/2$, we deduce that
\begin{equation}\label{e:w_n_ps_vs_nature}
\sum_{n\sim N} \frac{n^{1-\gamma}}{\gamma} 1_{n \in \PS_{1/\gamma}} \omega_n = \sum_{n \sim N}\omega_n + O(\log N) + \sum_{n \sim N} \frac{n^{1-\gamma}}{\gamma} \omega_n \Big(\psi(-(n+1)^{\gamma}) - \psi(-n^{\gamma})\Big).   
\end{equation}
Let $H  = N^{3\theta_d}$. By Lemma~\ref{approx-sawtooth}, we may approximate $\psi(x)$ by 
$$
\psi^*(x) = \sum_{1 \leq |h| \leq H} a_h e(hx), 
$$
where $|a_h| \ll |h|^{-1}$, such that for any real $x$,
$$
|\psi(x) - \psi^*(x)| \leq \sum_{|h| \leq H} b_h e(hx),
$$
where $|b_h| \ll H^{-1}$. Let $E(x) = \psi(x) - \psi^*(x)$.  Thus for $u \in \{0,1\}$ we have
\begin{equation}\label{eq:w_n_ps}
\sum_{n \sim N} \frac{n^{1-\gamma}}{\gamma}\omega_n \psi(-(n+u)^{\gamma}) = 
\sum_{1 \leq |h| \leq H} a_h \sum_{n \sim N} \frac{n^{1-\gamma}}{\gamma}\omega_n e(-h(n+u)^{\gamma}) + E,
\end{equation}
where the error term $E$ is
$$
E = \sum_{n \sim N} \frac{n^{1-\gamma}}{\gamma}\omega_n \Big(\psi(-(n+u)^{\gamma}) - \psi^*(-(n+u)^{\gamma})\Big).
$$
We have
$$
E \ll N^{1-\gamma}(\log N)\sum_{n \sim N} \sum_{|h| \leq H} b_h e(-h(n+u)^{\gamma})  \ll \frac{N^{1-\gamma}}{H} (\log N) \sum_{|h| \leq H} \Big|\sum_{n \sim N} e(-h(n+u)^{\gamma})\Big|
$$
The function $f(x) = -h(x+u)^{\gamma}$ satisfies 
$$
|f''(x)| \asymp |h| (x+u)^{\gamma-2} \asymp |h| N^{\gamma-2}
$$
for $x \sim N$. By van der Corput's inequality (Lemma \ref{van-der-corput} with $k=2$), we obtain for $h \neq 0$ the bound
$$
\sum_{n \sim N} e(-h(n+u)^{\gamma}) \ll N (|h|N^{\gamma-2})^{1/2} +  (|h|N^{\gamma-2})^{-1/2} \ll N^{\gamma/2} |h|^{1/2} + N^{1-\gamma/2} |h|^{-1/2}. 
$$
The $h=0$ term contributes $O(N^{2-\gamma}H^{-1}\log N)$ to $E$. Hence, after summing over $h$, we get
$$
E  \ll (N^{2-\gamma}H^{-1} + N^{1-\gamma/2}H^{1/2} + N^{2-3\gamma/2} H^{-1/2})\log N \ll N^{1-\theta_d}
$$
using our choice $H = N^{3\theta_d}$ and the assumption that $1-\gamma < \theta_d$. In view of \eqref{e:w_n_ps_vs_nature} and \eqref{eq:w_n_ps}, it thus suffices to show that
$$
\Big|\sum_{1 \leq |h| \leq H} a_h \sum_{n \sim N} \frac{n^{1-\gamma}}{\gamma} \omega_n (e(-h(n+1)^{\gamma}) - e(-hn^\gamma))\Big| \ll_{d,D} \delta^{-O_{d,D}(1)} N^{1-c_{d,D}}.
$$
Let $S$ be the sum on the left-hand side above, which can be written as
$$
S = \sum_{1 \leq |h| \leq H} a_h \sum_{n\sim N} B_h(n) \Lambda(n) e(-hn^\gamma) F(g(n)\Gamma),
$$
where
$$
B_h(n) = \frac{n^{1-\gamma}}{\gamma}(e(-hD(n))-1), \qquad D(n) = (n+1)^{\gamma}-n^{\gamma}.
$$
The smooth weight $B_h$ can be easily removed by partial summation as follows. Let
$$
S_h(t) = \sum_{N < n \leq t} \Lambda(n) e(-hn^\gamma) F(g(n)\Gamma).
$$
Then by partial summation we have
$$
S = \sum_{1 \leq |h| \leq H} a_h \Big(B_h(2N)S_h(2N)- \int_N^{2N} S_h(t)B_h'(t) dt\Big).
$$
One can compute that $D(x)\ll_{\gamma} x^{\gamma-1}$ and $D'(x) \ll_{\gamma}x^{\gamma-2}$. Using $|e(x)-1| \ll |x|$, we get
$$
|B_h(x)| \ll_{\gamma} x^{1-\gamma} |h|x^{\gamma-1} \ll_{\gamma} |h|.
$$
Furthermore, we have
$$
B_h'(x) = \frac{1-\gamma}{\gamma} x^{-\gamma}(e(-hD(x))-1) -\frac{2\pi ih}{\gamma} x^{1-\gamma} D'(x) e(-hD(x)).
$$
Hence $|B_h'(x)|\ll_{\gamma} |h|x^{-1}$. It follows that
$$
S \ll_{\gamma} \sum_{1 \leq |h| \leq H} |ha_h| \sup_{N<t\leq 2N} |S_h(t)| \ll_{\gamma} H \sup_{1 \leq |h| \leq H} \sup_{N<t\leq 2N} |S_h(t)|.
$$
By Proposition \ref{prop:supp-smooth-prime-twists} applied to such $S_h(t)$, we have
$$
S \ll_{d,D,\gamma} H \delta^{-O_{d,D}(1)} N^{1-1/(1000d\cdot 2^d)}.
$$
The conclusion follows by our choice of $H$.

\subsection{Gowers uniformity of fractional power phases}\label{sec:nil-orthogality2}

As mentioned in the paragraph following Proposition \ref{prop:supp-smooth-prime-twists}, we need to establish Gowers uniformity of the function $n\mapsto e(hn^{\gamma})$, which is the content of the following lemma.

\begin{lemma}[Gowers norms of fractional power phases]\label{lem:supp-fractional-gowers}
Fix $s\geq2$ and $0<\gamma<1$. Let $I=(N,2N]\cap\Z$ and let $J\subset I$ be an interval. Suppose that
\[
 N^{1/6}\leq |A|N^\gamma\leq N^{5/2}.
\]
Then the $U^s(I)$-norm of the function $n\mapsto 1_J(n) e(An^{\gamma})$ is  $O_{s,\gamma}(N^{-1/(50s\cdot 2^s)})$.
\end{lemma}

\begin{proof}
By the definition of the $U^s(I)$-norm, it suffices to prove that
$$
\sum_{x,h_1,\cdots,h_s \in \Z} \prod_{\omega \in \{0,1\}^s}  1_J(x+\omega \cdot \mathbf h) \cdot\mathcal C^{|\omega|}e(A (x+\omega\cdot\mathbf h)^{\gamma}) \ll N^{s+1-1/(50s)}.
$$
The contribution from those $|h_i| \leq N^{1-1/(50s)}$ is negligible, and hence we may restrict to those terms with $N^{1-1/(50s)} \leq |h_i| \leq N$ for every $i$. For such $\ve{h} = (h_1,\cdots,h_s)$, we will show that
$$
\sum_{x \in J_{\mathbf h}} e(\Phi_{\mathbf h}(x)) \ll N^{1-1/50},
$$
where $J_{\mathbf h}$ is the interval containing those $x \in \Z$ such that $x + \omega\cdot\mathbf h \in J$ for all $\omega$, and 
$$
\Phi_{\mathbf h}(x) = A\sum_{\omega \in \{0,1\}^s} (-1)^{|\omega|} (x+\omega\cdot \mathbf{h})^{\gamma}.
$$
Let $F(x) = x^{\gamma}$. Then we can write
\[
 \Phi_{\ve h}(x)=(-1)^sA \cdot\Delta_{h_1}\cdots\Delta_{h_s}F(x),
\]
where $\Delta_h$ is the difference operator defined by $\Delta_hF(x) = F(x+h)-F(x)$. By the fundamental theorem of calculus, we have
$$
\Phi_{\ve h}(x) = (-1)^s A (h_1\cdots h_s) \int_{[0,1]^s} F^{(s)}(x+\ve{t}\cdot\ve{h}) d\ve{t},
$$
where $F^{(s)}$ denotes the $s$th-order derivative of $F$. Differentiating the formula above three times and using the formula for $F^{(s+3)}$, we obtain
$$
 \Phi_{\ve h}(x)''' = (-1)^s A(h_1\cdots h_s) \gamma(\gamma-1)\cdots (\gamma-s-2)  \int_{[0,1]^s} (x+\ve{t}\cdot\ve{h})^{\gamma-s-3} d\ve{t}.
$$
For $x \in J_{\ve{h}}$, we have $x + \ve{t} \cdot \ve{h} \in J$ for all $\ve{t}$, and hence $x + \ve{t} \cdot\ve{h}\in (N, 2N]$ always. It follows that
$$
|\Phi_\ve{h}(x)'''| \asymp_{s,\gamma} |A(h_1\cdots h_s)| N^{\gamma-s-3}
$$
for all $x \in J_\ve{h}$. Thus we may apply Van der Corput's third derivative estimate (Lemma \ref{van-der-corput} with $k=3$) to conclude that
$$
\sum_{x \in J_{\mathbf h}} e(\Phi_{\mathbf h}(x)) \ll N \Delta^{1/6} + N^{1/2} \Delta^{-1/6},
$$
where 
$$
\Delta = |A(h_1\cdots h_s)| N^{\gamma-s-3} \in \Big[|A| N^{\gamma-3-1/50}, |A| N^{\gamma-3}\Big].
$$
The desired upper bound $\ll N^{1-1/50}$ follows after using the hypotheses on $|A|N^{\gamma}$.
\end{proof}

\begin{remark}
Lemma \ref{lem:supp-fractional-gowers} will be applied with $|A|N^{\gamma}$ a bit larger than $N^2$, hence the need to use the third derivative estimate in the proof. In fact, one can establish a more general version for $|A|N^{\gamma} = N^{O(1)}$ by taking higher derivative estimates instead.
\end{remark}

Lemma \ref{lem:supp-fractional-gowers} will be used in conjunction with the following fact about nilsequences. See \cite[Lemma A.2]{Matthiesen2018}, which attributes the following estimate to Green. If $G/\Gamma$ is a filtered nilmanifold of degree $d$ and dimension $D$, and complexity at most $1/\delta$ for some $\delta \in (0,1)$, then for any polynomial sequence $g \in \Poly(\Z\rightarrow G)$ and $F \in C^{\infty}(G/\Gamma)$ we have
\begin{equation}\label{eq:supp-dual}
 \|F(g(\cdot)\Gamma)\|_{U^{d+1}(J)^*}
 \ll \delta^{-O_{d,D}(1)} \|F\|_{W^{2^dD}},
\end{equation}
where $U^{d+1}(J)^*$ denotes the dual of the $U^{d+1}(J)$-norm.

\subsection{Proof of Proposition \ref{prop:supp-smooth-prime-twists}}\label{sec:nil-orthogonality3}

We may assume that $N$ is large enough. By Vaughan's identity with $U=V=N^{1/4}$:
$$
 \Lambda=\mu_{\leq U}*\log+\Lambda_{\leq V}
 -\mu_{\leq U}*1*\Lambda_{\leq V}
 +\mu_{>U}*1*\Lambda_{>V},
$$
it suffices to prove the following type I and type II estimates. 

\begin{lemma}\label{lem:type-I}
Let the notations and assumptions be as in the statement of Proposition \ref{prop:supp-smooth-prime-twists}. For $M \leq N^{1/2}$, $L\in [N/(2M), 2N/M]$, and $|b_m| \leq 1+\log m$, we have
$$
\Big|\sum_{\substack{\ell m\in J \\ \ell\sim L,m\sim M}} \psi(\ell)b_m e(h(\ell m)^{\gamma}) F(g(\ell m)\Gamma)\Big| \ll \delta^{-O_{d,D}(1)} N^{1-1/(600d\cdot 2^d)}
$$
for either $\psi(\ell)=1$ or $\psi(\ell)=\log\ell$.
\end{lemma}

\begin{proof}
Let $I$ be the left-hand side. Then
\begin{equation}\label{eq:type-I1}
I \ll (\log N) \sum_{m\sim M} \Big|\sum_{\substack{\ell\sim L\\ \ell m\in J}} \psi(\ell) e(h(\ell m)^{\gamma}) F(g(\ell m)\Gamma)\Big|.
\end{equation}
For fixed $m \sim M$, the inner sum over $\ell$ above is of the form
$$
\sum_{\ell \in J_m} \psi(\ell) e(A_m \ell^{\gamma}) F(g_m(\ell)\Gamma),
$$
where $J_m = \{\ell\sim L: \ell m \in J\} \subset (L, 2L]$, $A_m = hm^{\gamma}$, and $g_m(\ell) = g(\ell m)$. Since $L \geq N^{1/2}/2$, we have
$$
L^{1/6} \leq |A_m|\ell^{\gamma} = |h|(\ell m)^{\gamma} \leq 2N^{\gamma+1/6} \leq L^{5/2}.
$$
By Lemma \ref{lem:supp-fractional-gowers} with $I = (L, 2L] \cap \Z$, the $U^{d+1}(I)$-norm of the function $\ell\mapsto 1_{J_m}(\ell) e(A_m\ell^{\gamma})$ is $O(L^{-1/(200d\cdot 2^d)})$. By \eqref{eq:supp-dual} and the hypothesis $\|F\|_{W^m} \leq 1/\delta$, the dual $U^{d+1}(I)^*$-norm of the nilsequence $\ell\mapsto F(g_m(\ell)\Gamma)$ is $O(\delta^{-O_{d,D}(1)})$. It follows immediately that 
$$
\Big|\sum_{\ell \in J_m}  e(A_m \ell^{\gamma}) F(g_m(\ell)\Gamma)\Big| \ll \delta^{-O_{d,D}(1)} L^{1-1/(200d\cdot 2^d)} \ll \delta^{-O_{d,D}(1)} LN^{-1/(400d\cdot 2^d)}.
$$
With the weight $\psi(\ell) = \log \ell$, the same bound (with an extra factor of logarithm) holds after partial summation. Plugging this into \eqref{eq:type-I1} concludes the proof.
\end{proof}

\begin{lemma}\label{lem:type-II}
Let the notations and assumptions be as in the statement of Proposition \ref{prop:supp-smooth-prime-twists}. For $L,M \geq N^{1/4}$ with $LM\asymp N$, $|a_{\ell}| \leq \log N$, and $|b_m| \leq \log N$, we have
$$
\Big|\sum_{\substack{\ell m\in J \\ \ell\sim L,m\sim M}} a_{\ell}b_m e(h(\ell m)^{\gamma}) F(g(\ell m)\Gamma)\Big| \ll \delta^{-O_{d,D}(1)} N^{1-1/(1000d\cdot 2^d)}.
$$
\end{lemma}

\begin{proof}
Without loss of generality, assume that $L \leq M$. We may assume that $LM \in [N/4, 4N]$, so that $M \geq N^{1/2}/2$. Let $I$ be the left-hand side. Then
$$
I \ll (\log N)\sum_{m\sim M} \Big|\sum_{\substack{\ell\sim L \\ \ell m \in J}} a_{\ell} e(h(\ell m)^{\gamma}) F(g(\ell m)\Gamma)\Big|.
$$
By Cauchy-Schwarz and expanding the square, we have
\begin{equation}\label{eq:type-II1}
I^2 \ll (\log N)^4 M \sum_{\ell_1,\ell_2\sim L} \Big|\sum_{\substack{m \sim M \\ \ell_1m,\ell_2m\in J}} e(h(\ell_1^{\gamma}-\ell_2^{\gamma})m^{\gamma}) F(g(\ell_1m)\Gamma) \overline{F(g(\ell_2m)\Gamma)}\Big|.
\end{equation}
The total contribution from the diagonal terms with $\ell_1=\ell_2$ is clearly negligible. Now consider those terms with $\ell_1\neq \ell_2$. The inner sum over $m$ is of the form
$$
\sum_{m \in J_{\ell_1,\ell_2}} e(A_{\ell_1,\ell_2} m^{\gamma}) \wt{F}(g_{\ell_1,\ell_2}(m)\wt{\Gamma}),
$$
where $J_{\ell_1,\ell_2} \subset (M,2M]$, $A_{\ell_1,\ell_2} = h(\ell_1^{\gamma}-\ell_2^{\gamma})$, $\wt{F}$ is the function on the product nilmanifold $\wt{G}/\wt{\Gamma} := (G\times G)/(\Gamma\times \Gamma)$ defined by $\wt{F}(x, y) = F(x)\overline{F(y)}$, and $g_{\ell_1,\ell_2}$ is the polynomial sequence defined by $g_{\ell_1,\ell_2}(m) = (g(\ell_1m), g(\ell_2m))$. Since
$$
|A_{\ell_1,\ell_2}m^{\gamma}| \geq |(\ell_1^{\gamma}-\ell_2^{\gamma}) M^{\gamma}| \gg \gamma L^{\gamma-1}M^{\gamma} \geq \gamma M^{2\gamma-1} \geq M^{1/6},
$$
and
$$
|A_{\ell_1,\ell_2}m^{\gamma}| \ll |h| (LM)^{\gamma} \ll N^{\gamma+1/6} \leq M^{5/2},
$$
Lemma \ref{lem:supp-fractional-gowers} with $I = (M, 2M]\cap\Z$ implies that the $U^{d+1}(I)$-norm of the function $m\mapsto 1_{J_{\ell_1,\ell_2}}(m)e(A_{\ell_1,\ell_2}m^{\gamma})$ is $O(M^{-1/(200d\cdot 2^d)})$. By \eqref{eq:supp-dual}, the dual $U^{d+1}(I)^*$-norm of the nilsequence $m\mapsto \wt{F}(g_{\ell_1,\ell_2}(m)\wt\Gamma)$ is $O(\delta^{-O_{d,D}(1)})$. It follows that 
$$
\sum_{m \in J_{\ell_1,\ell_2}} e(A_{\ell_1,\ell_2} m^{\gamma}) \wt{F}(g_{\ell_1,\ell_2}(m)\wt{\Gamma}) \ll \delta^{-O_{d,D}(1)}M^{1-1/(200d\cdot 2^d)} \ll \delta^{-O_{d,D}(1)}MN^{-1/(400d\cdot 2^d)}.
$$
Plugging this into \eqref{eq:type-II1} concludes the proof.
\end{proof}

\section{Pseudorandom majorants for Piatetski-Shapiro primes}\label{sec:pseudo}

A further ingredient in the proof of Theorem \ref{thm:lin-eq-PS} is the construction of a pseudorandom majorant for the Piatetski-Shapiro primes that is sufficiently uniform to make the transference principle work. This construction has been obtained by Li and Pan \cite{LiPan} and we refine their analysis to achieve a quantitative improvement.

\begin{definition}[Pseudorandom majorant]\label{def:majorant}
Let $\nu:\Z_N\to[0,\infty)$ be a measure, and let $m_0, d_0, L_0$ be positive integer parameters. Then we say that $\nu$ satisfies the $(m_0,d_0,L_0)$-linear forms condition if the following holds: given $1 \leq d \leq d_0, 1 \leq t \leq m_0,$ and any finite complexity system $\Psi=(\psi_1, \dots, \psi_t)$ of affine-linear forms on $\Z^d$ with linear coefficients bounded in magnitude by $L_0$ and arbitrary constant terms, we have  
$$
\lim_{N \to \infty} {\frac{1}{N^d}}\sum_{{\bf x} \in \Z_N^d} \prod_{i \in [t]}\nu(\psi_i({\bf x}))=1.
$$
We call $\nu$ an $m$-pseudorandom measure if it satisfies  the $(2^{m-1}m, 2m, m)$-linear forms condition.
\end{definition}

\begin{proposition}[Existence of pseudorandom majorant]\label{prop:majorant}
Let $m$ be a positive integer and let $C_m$ be a sufficiently large constant in terms of $m$. Suppose that $1-\gamma < 2^{-Cm}$ for some sufficiently large absolute constant $C$. Let $X \geq 3$, $N \in [C_mX, 2C_mX]$ be prime,  $W = \prod_{p \leq w}p$ with $w = w(X)$  a slowly growing function, and {$1\leq b\leq W$ with} $(b,W)=1$. There exists an $m$-pseudorandom measure $\nu$ on $\Z_N$ such that
$$
\nu(n) \gg_{m} \frac{\phi(W)}{W}(Wn+b)^{1-\gamma} \log N
$$
for $n \in [X^{0.9}, X]$ with $Wn+b \in \PS_{1/\gamma} \cap \mathbb{P}$.
\end{proposition}

This is a refinement and a quantitative improvement of the pseudorandom majorant constructed by Li and Pan \cite{LiPan}, where the majorization condition holds in a short interval $n \in [(1-c_m)X, X]$ and the assumption required for $1-\gamma$ is triply exponential in $m$. For our purposes of counting asymptotically the number of $k$-APs (for example) in $\PS_{1/\gamma} \cap [1, X]$, it is necessary to have information about $\Lambda_{\gamma}$ and its majorant on a global scale, since the different terms of the $k$-APs could come from different short intervals. 

\begin{remark}\label{rmk:majorant}
The pseudorandom majorant $\nu$ for (weighted and $W$-tricked) $\Lambda_{\gamma}$ is required in two places: the generalized von Neumann theorem (see \cite[Proposition 7.1]{GT-linear}) and the inverse theorem for Gowers norms (see \cite[Proposition 9.4]{MSTT}). In \cite{GT-linear}, the pseudorandom property for $\nu$ additionally requires a ``correlation condition", but is not used in the proof of \cite[Proposition 7.1]{GT-linear}. For the inverse theorem, \cite[Proposition 9.4]{MSTT} removes the correlation condition and only requires $\nu-1$ to have small Gowers norms. One can check that if $\nu$ is $s$-pseudorandom in the sense of Definition \ref{def:majorant}, then $\|\nu-1\|_{U^s} = o(1)$.
\end{remark}

In this section we prove Proposition \ref{prop:majorant}. The key new ingredient is the following exponential sum estimate.

\begin{lemma}\label{l:multi_linear_form}
Let $s\geq1$, $A\geq1$, $c_0>0$, and
$$
0<1-\gamma\leq\frac1{100As^3}.
$$
Let $a=As(1-\gamma)$. Let $N$ be sufficiently large and let $\mathcal I\subset\Z$ be an interval of integers. Suppose that the affine-linear forms $\psi_i(x)=\alpha_i x+\beta_i$, with $\alpha_i\in\Z\setminus\{0\}$ and $\beta_i\in\R$, satisfy
$$
N^{1-a}\leq\psi_i(x)\leq N\quad(x\in\mathcal I),\qquad
|\alpha_i\beta_j-\alpha_j\beta_i|\geq c_0N^{1-a}\quad(i\ne j).
$$
Let $h_i\in\R$ satisfy $1\leq|h_i|\leq N^a$, and let $F(x)=\sum_{i=1}^s h_i\psi_i(x)^\gamma$. Then for all $\theta\in\R$ we have
$$
\left|\sum_{x\in\mathcal I}e(F(x)+\theta x)\right|
\ll_{A,s,\gamma,c_0}N^{1-1/(10^5s^2)}.
$$
\end{lemma}

\begin{proof}
From the assumption that $N^{1-a} \leq \psi_i(x) \leq N$ for all $x \in \mathcal I$ one can deduce that $|\mathcal{I}| \leq 1+N/|\alpha_i|$ for each $i$. In particular, we have $|\mathcal{I}| \leq N$. We may assume that $|\mathcal{I}| \geq N^{1-1/(10s)}$, since otherwise the conclusion follows trivially. We may also assume that $|\alpha_i| \leq N^{1/(100s^2)}$ for each $i$, since otherwise $|\mathcal{I}| \leq N^{1-1/(100s^2)}$ and the conclusion also follows trivially.

Partition $\mathcal I$ into subintervals of length $\asymp N^{1-1/(10s)}$. It suffices to prove that for each such subinterval of the form $\{x_0, x_0+1,\cdots, x_0+M-1\}$, where $M \asymp N^{1-1/(10s)}$, we have
$$
\Big|\sum_{0 \leq m < M} e(F(x_0+m) + \theta m)\Big| \ll_{A,s,\gamma,c_0} MN^{-1/(10^5s^2)}.
$$
By the Taylor expansion of $F(x_0+\cdot)$, we have
\begin{equation}\label{eq:PS-taylor-polynomial}
F(x_0+m)=\sum_{j=0}^{100s} e_jm^j +O_{s,\gamma}(N^{-2}),
\end{equation}
where
$$
e_j=\frac{F^{(j)}(x_0)}{j!} = \binom\gamma j\sum_{i=1}^s h_i\psi_i(x_0)^{\gamma} \Big(\frac{\alpha_i}{\psi_i(x_0)}\Big)^j,
$$
and the estimate for the Taylor remainder comes from
$$
|F^{(100s+1)}(x)| M^{100s+1} \ll_{s,\gamma}N^{a+\gamma} \Big(\frac{N^{1/(100s^2)}}{N^{1-a}}\Big)^{100s+1} N^{(1-1/(10s))(100s+1)} \ll_{s,\gamma} N^{-2}
$$
for $x_0 \leq x \leq x_0+M-1$.
By \eqref{eq:PS-taylor-polynomial}, it suffices to prove that
\begin{equation}\label{eq:Weyl}
\Big|\sum_{0 \leq m < M} e\Big(\sum_{j=0}^{100s} e_jm^j + \theta m\Big)\Big| \ll_{A,s,\gamma,c_0} MN^{-1/(10^5s^2)}.
\end{equation}

We next find a coefficient $e_j$, with $2\leq j\leq s+1$, which is not too small. We make a change of variables:
$$
S_j = \binom{\gamma}{j}^{-1}N^je_j,\qquad v_i = h_i\psi_i(x_0)^{\gamma}, \qquad c_i = \frac{N\alpha_i}{\psi_i(x_0)}.
$$
Then we have $S_j = \sum_{1 \leq i \leq s} v_ic_i^j$, which can be written in matrix form as
$$
\begin{pmatrix}
S_2\\ S_3\\ \vdots\\ S_{s+1}
\end{pmatrix}
=
\begin{pmatrix}
1&1&\cdots&1\\
c_1&c_2&\cdots&c_s\\
\vdots&\vdots&\ddots&\vdots\\
c_1^{s-1}&c_2^{s-1}&\cdots&c_s^{s-1}
\end{pmatrix}
\begin{pmatrix}
v_1c_1^2\\v_2c_2^2\\\vdots\\v_sc_s^2
\end{pmatrix}.
$$
Let $\ve{M}$ be the Vandermonde matrix appearing above, and let $\ve{M}^{-1} = (\lambda_{ik})$. Note that 
$$
1 \leq |c_i| \leq N^{a+1/(100s^2)}, \qquad |c_i-c_j| = N \frac{|\alpha_i\psi_j(x_0)-\alpha_j\psi_i(x_0)|}{\psi_i(x_0)\psi_j(x_0)} \geq c_0N^{-a}
$$
for $i \neq j$. From the identity
$$
\sum_{k=0}^{s-1}\lambda_{ik}z^k = \prod_{\ell\neq i} \frac{z-c_{\ell}}{c_i-c_{\ell}}
$$
and by examining the coefficient of $z^k$, we obtain the bound
$$
\sum_{k=0}^{s-1}|\lambda_{ik}|
\leq\frac{\prod_{\ell\ne i}(1+|c_\ell|)}
{\prod_{\ell\ne i}|c_i-c_\ell|}
\ll_{s,c_0}N^{(s-1)(2a+1/(100s^2))}.
$$
Since
$$
\Big|\sum_{k=0}^{s-1} \lambda_{ik} S_{k+2}\Big| = |v_ic_i^2| \geq |N^2 \psi_i(x_0)^{\gamma-2}| \geq N^{\gamma},
$$
it follows that there exists $2 \leq j \leq s+1$ such that
$$
|S_j|
\gg_{s,c_0}N^{\gamma-(s-1)(2a+1/(100s^2))}.
$$
Combining this with the trivial upper bound
$$
|S_j|\leq sN^{\gamma+a+j(a+1/(100s^2))}
$$
and using the relation between $S_j$ and $e_j$,
we conclude that
$$
N^{\gamma-j-1/(20s)}\ll_{s,\gamma,c_0}|e_j|
\ll_{s,\gamma,c_0}N^{\gamma-j+1/(20s)}.
$$
Thus $|e_j|$ may be well approximated by $1/q$, where $q=\lfloor |e_j|^{-1}\rfloor$, such that $||e_j| - 1/q| \leq 1/q^2$. Moreover, we have
$$
N^{j-\gamma-1/(20s)} \ll_{s,\gamma,c_0} q \ll_{s,\gamma,c_0} N^{j-\gamma+1/(20s)}.
$$
To prove \eqref{eq:Weyl}, the left-hand side can be bounded using \cite[Theorem 5]{B17} and the above Diophantine approximation of $e_j$ by
$$
\ll_{s,\gamma,c_0,\eps} M^{1+\eps} (q^{-1} + M^{-1} + qM^{-j})^{1/(10^4s^2)}.
$$
This gives the desired estimate.
\end{proof}

\begin{proof}[Proof of Proposition~\ref{prop:majorant}]
We follow the argument of Li and Pan \cite{LiPan}, using Lemma~\ref{l:multi_linear_form} for the exponential sums.  Let $s_0=m2^{m-1}$
and suppose that $0<1-\gamma\leq1/(10^7s_0^4)$. Let $R = N^{1/(10^7s_0^3)}$. Define
$$
\Lambda_R(n)=\sum_{\substack{d\mid n\\d\leq R}}\mu(d)\log\frac Rd,
\qquad \nu_R(n)=\frac{\phi(W)}W\frac{\Lambda_R(n)^2}{\log R}, 
\qquad T(n) = \frac{n^{1-\gamma}}{\gamma} 1_{n \in \PS_{1/\gamma}}.
$$
Recall that $N \in [C_mX, 2C_mX]$.
Identify $\Z_N$ with $[1,N]$ and define
$$
\nu(n)=\begin{cases}
(\nu_RT)(Wn+b),&X^{0.9}\leq n\leq N,\\
1,&1\leq n<X^{0.9}.
\end{cases}
$$
This is indeed a majorant: If $n \in [X^{0.9},X]$ and $Wn+b$ is a Piatetski-Shapiro prime in $\PS_{1/\gamma}$, then $\nu_R(Wn+b) \gg_m (\phi(W)/W)\log N$ and
$$
\nu(n) \gg_m \frac{\phi(W)}{W} (Wn+b)^{1-\gamma} \log N,
$$
as desired.

Let $\Psi = (\psi_1,\cdots,\psi_t)$ be a finite complexity system of $t \leq s_0$ affine-linear forms in $d \leq 2m$ variables, as in the definition of an $m$-pseudorandom majorant. Our goal is to prove that
$$
\frac{1}{N^d} \sum_{\ve{x} \in \Z_N^d} \prod_{1 \leq i\leq t}\nu(\psi_i(\ve{x})) = 1+o(1).
$$
We view $\psi_i(\ve{x})$ as affine linear forms with integer coefficients. At the moment the values of $\psi_i(\ve{x})$ are reduced modulo $N$. We may divide $[N]^d$ into $O_m(1)$ convex bodies, such that on each of these convex bodies the values of $\psi_i(\ve{x})$ lie in an interval of length $N$. After translating each $\psi_i(\ve{x})$ appropriately, the task is then reduced to showing that
\begin{equation}\label{eq:majorant-restricted}
\sum_{\ve{x} \in K \cap [N]^d} \prod_{1 \leq i\leq t} (\nu_RT)(W\psi_i(\ve{x})+b)  = |K \cap [N]^d| + o(N^d)
\end{equation}
for any convex body $K$ on which $1 \leq \psi_i(\ve{x}) \leq N$ for each $i$. Write $\ve{x} = (x_1,\ve{y})$ for $x_1 \in [N]$ and $\ve{y} \in [N]^{d-1}$, and write
$$
W\psi_i(\ve{x}) + b = \alpha_i x_1 + \beta_i(\ve{y})
$$
for some $\alpha_i \in \Z$ and some affine linear form $\beta_i$. We have $|\alpha_i|\ll_m W$ for each $i$. We may assume that $\alpha_i \neq 0$ for each $i$ (otherwise make a change of variable at the beginning of the proof). Note that the number of $\ve{x} \in K \cap [N]^d$ with $\psi_i(\ve{x})<N^{1-4s_0(1-\gamma)}$ is $O(N^{d-4s_0(1-\gamma)})$, and the number of $\ve{x} = (x_1,\ve{y}) \in K\cap [N]^d$ with $|\alpha_i\beta_j(\ve y) - \alpha_j\beta_i(\ve y)|< N^{1-4s_0(1-\gamma)}$ is also $O(N^{d-4s_0(1-\gamma)})$, since $\alpha_i\beta_j(\ve{y}) - \alpha_j\beta_i(\ve{y})$ has a nonzero linear coefficient from the assumption that the system has finite complexity. Thus the total contribution from these exceptional points to \eqref{eq:majorant-restricted} is negligible. The set of non-exceptional points is a union of $O_m(1)$ convex bodies. Hence it suffices to establish \eqref{eq:majorant-restricted} for each of these convex bodies, which we continue to call $K$. Hence for $\ve{x} = (x_1,\ve{y}) \in K \cap [N]^d$ we have
\begin{equation}\label{eq:majorant-separation}
\psi_i(\ve{x})\geq N^{1-4s_0(1-\gamma)}\quad(1\leq i\leq t),\qquad
|\alpha_i\beta_j(\ve y) - \alpha_j\beta_i(\ve y)|\geq N^{1-4s_0(1-\gamma)}\quad(i\ne j).
\end{equation}

We next reduce \eqref{eq:majorant-restricted} to an exponential sum estimate. By \eqref{eq:PS-indicator} and Lemma \ref{approx-sawtooth} with $H = \lfloor N^{2s_0(1-\gamma)}\rfloor$, we can approximate $T(n)$ by 
$$
T_H(n)=1+\frac{n^{1-\gamma}}{\gamma}\sum_{0<|h|\leq H}a_h
\big(e(-h(n+1)^\gamma)-e(-hn^\gamma)\big),
$$
where $|a_h| \ll |h|^{-1}$, such that
\begin{equation}\label{eq:T-TH}
|T(n)-T_H(n)|\ll_\gamma n^{-1}
+n^{1-\gamma}\big(K_H(-n^\gamma)+K_H(-(n+1)^\gamma)\big),
\end{equation}
where $K_H = \sum_{|h|<H}b_h e(hx)$ with $|b_h| \ll H^{-1}$ is the nonnegative error polynomial in Lemma~\ref{approx-sawtooth}. 

We first establish \eqref{eq:majorant-restricted} with $T$ replaced by $T_H$. By the definition of $T_H$, the left-hand side can be decomposed as the main term
$$
\sum_{\ve{x}\in K \cap [N]^d} \prod_{1 \leq i \leq t}\nu_R(W\psi_i(\ve{x})+b),
$$
plus other terms of the form
$$
\pm \sum_{\ve{x}\in K \cap [N]^d} \prod_{1 \leq i \leq t} \nu_R(W\psi_i(\ve{x})+b) \prod_{i \in I} \frac{(W\psi_i(\ve{x})+b)^{1-\gamma}}{\gamma} a_{h_i} e(-h_i(W\psi_i(\ve{x})+b+u_i)^{\gamma})
$$
for some nonempty subset $I \subset [t]$, $0 < |h_i| \leq H$ and $u_i\in\{0,1\}$ for $i\in I$.

The main term above is $|K \cap [N]^d| + o(N^d)$ by the sieve calculation in \cite{LiPan}, based on \cite[Proposition 9.5]{GT-kAP}. It remains to show that each of the error terms above is $o(N^d)$. For ease of notation, we write $p_i(\ve{x}) = W\psi_i(\ve{x})+b$. Following the arguments in \cite[Section 4]{LiPan}, we first expand the $\Lambda_R^2$ factor to write the error term above as
$$
\begin{aligned}
& \Big(\prod_{i\in I}a_{h_i}\Big)\Big(\frac{\phi(W)}{W\log R}\Big)^t\sum_{d_i,d_i'\leq R} \Big(\prod_{i=1}^t \mu(d_i)\mu(d_i')\log\frac{R}{d_i}\log\frac{R}{d_i'}\Big)\times \\
&\qquad \sum_{\substack{\ve{x} \in K \cap [N]^d \\ [d_i,d_i']\mid p_i(\ve{x})}} \Big(\prod_{i\in I} \frac{p_i(\ve{x})^{1-\gamma}}{\gamma}\Big) e\Big(-\sum_{i \in I} h_i(p_i(\ve{x})+u_i)^{\gamma}\Big).
\end{aligned}
$$
After using additive characters to detect the divisibility conditions $[d_i,d_i']\mid p_i(\ve{x})$ and using partial summation to deal with the weight $\prod_{i\in I} p_i(\ve{x})^{1-\gamma}/\gamma$,  it suffices to show for $\theta_i \in \R$ that 
\begin{equation}\label{eq:majorant-expsum}
\sum_{\ve{x} \in K \cap [N]^d} e\Big(-\sum_{i \in I} h_i(p_i(\ve{x})+u_i)^{\gamma} + \sum_{i=1}^t \theta_i p_i(\ve{x})\Big) = o\Big(\frac{N^d}{R^{2t} N^{(1-\gamma)t} (\log N)^{2t}}\Big).
\end{equation}
Recall that for $\ve{x}=(x_1,\ve{y})$ we have $p_i(\ve{x}) = \alpha_ix_1+\beta_i(\ve{y})$, where $\alpha_i\neq 0$ for each $i$ and \eqref{eq:majorant-separation} is satisfied. The sum in \eqref{eq:majorant-expsum} can be written as
$$
\sum_{\ve{y} \in [N]^{d-1}} \sum_{\substack{x_1 \in [N] \\ (x_1,\ve{y}) \in K}} 
 e\Big(-\sum_{i \in I} h_i(\alpha_ix_1+\beta_i(\ve{y})+u_i)^{\gamma} + \Big(\sum_{i=1}^t \theta_i\alpha_i\Big) x_1+ \sum_{i=1}^t\theta_i\beta_i(\ve{y})\Big).
 $$
By Lemma \ref{l:multi_linear_form} applied to the inner sum over $x_1$, one obtains the bound $O(N^{d-1/(10^5t^2)+o(1)})$ for the sum above. This leads to the desired estimate \eqref{eq:majorant-expsum} by our choice of $R$ and the assumption on $1-\gamma$. Hence \eqref{eq:majorant-restricted} holds with $T$ replaced by $T_H$.

It remains to prove that
$$
\sum_{\ve x\in K \cap [N]^d}\prod_i\nu_R(p_i(\ve{x}))
\Big(\prod_iT(p_i(\ve{x}))-\prod_iT_H(p_i(\ve{x}))\Big) = o(N^d).
$$
The left-hand side above can be decomposed into terms of the form
$$
\sum_{\ve x\in K \cap [N]^d}\prod_i\nu_R(p_i(\ve{x})) T_i(p_i(\ve{x})),
$$
where each $T_i$ is either $T_H$ or $T-T_H$, and $T_i = T-T_H$ for at least one index $i$. Fix this choice of $i$ and use the bound \eqref{eq:T-TH} for $T_i$. The contribution from the $O(n^{-1})$ term in the error from \eqref{eq:T-TH} is clearly negligible. Since $\|\nu_R\|_{\infty} \ll N^{o(1)}$, $\|T\|_{\infty}, \|T_H\|_{\infty} \ll N^{1-\gamma+o(1)}$, it suffices to bound
$$
N^{(1-\gamma)t+o(1)}\sum_{\ve x \in K \cap [N]^d} (K_H(-p_i(\ve{x})^\gamma) + K_H(-(p_i(\ve{x})+1)^\gamma)).
$$
Now use the Fourier expansion of $K_H(x) = \sum_{|h|<H} b_he(hx)$. The contribution from the $h=0$ term is negligible since $|b_0| \ll H^{-1}$. For $0 < |h|<H$ and $u_i\in\{0,1\}$,  from Lemma \ref{l:multi_linear_form} applied to the single linear form $x_1\mapsto \alpha_ix_1 + \beta_i(\ve{y})+u_i$ for fixed $\ve{y}$ (see \eqref{eq:majorant-expsum}), we obtain
$$
\sum_{\ve x \in K \cap [N]^d}  e(-h(p_i(\ve{x})+u_i)^\gamma) \ll N^{d-1/(10^5t^2)+o(1)}.
$$
This completes the proof.
\end{proof}

\section{Gowers uniformity and linear equations}\label{sec:deduction}

In Section \ref{sec:deduction1} we deduce Gowers uniformity of $\Lambda_{\gamma}-\Lambda$ (after $W$-trick) from our nilsequence estimate (Proposition \ref{prop:supp-BF}) using our pseudorandom majorant (Proposition \ref{prop:majorant}). In Section \ref{sec:deduction2} we deduce Theorem \ref{thm:lin-eq-PS} from Gowers uniformity and the existence of the pseudorandom majorant.  In Section \ref{sec:deduction3} we deduce Theorem \ref{t:improvement_li-pan} from the existence of the pseudorandom majorant using the transference principle.
 
\subsection{Gowers uniformity}\label{sec:deduction1}

\begin{proposition}[Gowers norm estimate]\label{prop:gowers-PS}
Let $s \geq 2$. Let $M \geq 3$, $W = \prod_{p \leq w}p$ with $w = w(M)$ a slowly growing function, and $1 \leq b \leq W$ with $(b,W)=1$. Then
$$
\|\frac{\phi(W)}{W}(\Lambda_{\gamma}-\Lambda)(W\cdot + b)\|_{U^s(M)} = o(1),
$$
provided that $1-\gamma$ is sufficiently small in terms of $s$.
\end{proposition}

\begin{proof}
In order to apply the inverse theorem \cite[Proposition 9.4]{MSTT}, we need a majorant $\nu$ for (a slightly modified version of) $(\phi(W)/(CW))(\Lambda_{\gamma}+\Lambda)(W\cdot+b)$ satisfying $\|\nu-1\|_{U^{2s}(M)} = o(1)$, where $C=C_s$ is some sufficiently large constant. Indeed, we may discard those $n < M^{0.9}$ and those $n$ with $Wn+b$ higher prime powers, since their contribution to the $U^s(M)$-norm is $o(1)$. Then we use the $2s$-pseudorandom majorant from Proposition \ref{prop:majorant} and the usual sieve majorant for the primes, and then take $\nu$ to be their average (see also \cite[Remark 9.7]{MSTT}).

By the inverse theorem \cite[Proposition 9.4]{MSTT}, it suffices to prove that for any filtered nilmanifold $G/\Gamma$ of degree at most $s-1$, any polynomial sequence $g \in \Poly(\Z\rightarrow G)$, and any Lipschitz function $F:G/\Gamma\rightarrow\C$, we have
$$
\frac{\phi(W)}{W}\sum_{n \leq M} (\Lambda_{\gamma}-\Lambda)(Wn+b) F(g(n)\Gamma) = o(M),
$$
where the rate of decay in $o(M)$ is allowed to depend on $G/\Gamma$ and the Lipschitz norm of $F$. This follows from
Proposition \ref{prop:supp-BF} by a standard calculation involving detecting the residue class $b\Mod{W}$ by additive characters.
\end{proof}

\subsection{Proof of Theorem~\ref{thm:lin-eq-PS}}\label{sec:deduction2}

We may restrict to those $\ve{n}$ such that $\psi_i(\ve{n})$ is prime for each $i$, since the contribution from higher prime powers can be shown to be negligible. We may restrict to those $\ve{n}$ with $\psi_i(\ve{n}) \geq X^{19/20}$ for each $i$, since the other terms make negligible contribution to the sum.

Let $w = w(X)\to\infty$ sufficiently slowly and $W=\prod_{p\leq w}p$. We apply the $W$-trick as in \cite[Section 5]{GT-linear}. Let
$$
\mathcal A=\{\ve a\in[W]^d:(\psi_i(\ve a),W)=1\text{ for all }i\}.
$$
If $\ve{n}$ contributes to the sum, then $\ve{n}\Mod{W} \in \mathcal A$. From \cite[equation (5.4)]{GT-linear} we have
$$
|\mathcal A|\ll_{d,t,L}W^d(\phi(W)/W)^t.
$$
For $\ve a\in\mathcal A$, write
$$
K_{\ve a}=(K-\ve a)/W,\qquad
\psi_i(W\ve n+\ve a)=W\widetilde\psi_{i,\ve a}(\ve n)+b_i(\ve a),
$$
where $b_i(\ve a)\in[W]$ is coprime to $W$. Let $M=\lceil CX/W\rceil$ with a  sufficiently large constant $C$ in terms of $d,t,L$. It suffices to prove for each $\ve{a} \in \mathcal A$ that
$$
\Big(\frac{\phi(W)}{W}\Big)^t\Big[\sum_{\ve{n} \in K_{\ve a}} \prod_{i=1}^t \Lambda_{\gamma}(W\wt\psi_{i,\ve a}(\ve{n})+b_i(\ve{a})) - \sum_{\ve{n} \in K_{\ve a}} \prod_{i=1}^t \Lambda(W\wt\psi_{i,\ve a}(\ve{n})+b_i(\ve{a}))\Big] = o_{d,t,L,\gamma}(M^d).
$$
After decomposing the left-hand side above, it suffices to prove that
$$
\sum_{\ve{n} \in K_{\ve a}} \prod_{i=1}^t f_i(\wt{\psi}_{i,\ve{a}}(\ve{n})) = o_{d,t,L,\gamma}(M^d),
$$
where each $f_i$ is either one of the two functions:
$$
\frac{\phi(W)}{W} \Lambda_{\gamma}(W\cdot+b), \qquad \frac{\phi(W)}{W}(\Lambda_{\gamma}-\Lambda)(W\cdot +b),
$$
where $b = b_i(\ve{a})$, and at least one $f_i$ is the second function. As in the proof of Proposition \ref{prop:gowers-PS}, each $|f_i|/C$ for some constant $C$ has a pseudorandom majorant $\nu_i$. The proof of Proposition \ref{prop:majorant} also applies to show that $\nu_1,\cdots,\nu_t$ satisfy the mixed linear forms estimates, and hence each $|f_i|/(Ct)$ can be majorized by the pseudorandom majorant $(\nu_1+\cdots+\nu_t)/t$.

By the generalized von Neumann theorem (\cite[Proposition 7.1]{GT-linear}) (after performing a normal-form reduction as in \cite[Section 4]{GT-linear}) and using the pseudorandom majorants from Proposition \ref{prop:majorant}, it suffices to prove that
$$
\Big\|\frac{\phi(W)}{W}(\Lambda_{\gamma}-\Lambda)(W\cdot +b)\Big\|_{U^s(M)} = o(1)
$$
for some $s$, which is precisely the Gowers uniformity in Proposition \ref{prop:gowers-PS}.

\subsection{Proof of Theorem \ref{t:improvement_li-pan}}\label{sec:deduction3}

We follow the transference argument in \cite{LiPan}. Let $X$ be large,
 let $w = w(X)\to\infty$ sufficiently slowly, and let
$W=\prod_{p\leq w}p$. Pick a prime $N\in[C_kX,2C_kX]$, where $C_k$ is a sufficiently large constant in terms of $k$.
For each reduced
residue $b\in[W]$, define on $\Z_N$
\[
 f_b(n)=\begin{cases} c_k\frac{\phi(W)}{W} (Wn+b)^{1-\gamma} \log N, & \text{if }
             n \in [X/2, X]\text{ and } Wn+b \in \PS_{1/\gamma} \cap \mathbb{P}, \\
             0, & \text{otherwise,}\end{cases}
\]
where $c_k>0$ is a sufficiently small constant in terms of $k$. Then the $k$-pseudorandom majorant $\nu = \nu_b$ on $\Z_N$ constructed in Proposition \ref{prop:majorant} satisfies $0 \leq f_b \leq \nu_b$. By the Piatetski-Shapiro prime number theorem \cite{PS}, we may choose $b$ such that
\[
 \sum_{n\in\Z_N}f_b(n)    \gg_k X \gg_k N.
\]
By \cite[Theorem 2.2]{LiPan}, we have
\[
 \sum_{x,d \in\Z_N} \prod_{j=0}^{k-1}f_b(x+jd)\gg_k N^2.
\]
The contribution from the $d=0$ term is clearly negligible. Thus there exists $x,d \in \Z_N$ with $d \neq 0$ such that $x, x+d,\cdots,x+(k-1)d$ are all in the support of $f_b$. Since all of them lie in $[X/2, X]$, they remain a $k$-term arithmetic progression in the integers, when viewing $x,d$ as integers in the natural way. Hence 
$$
Wx+b, W(x+d)+b, \cdots, W(x+(k-1)d)+b
$$ 
form a $k$-term arithmetic progression of Piatetski-Shapiro primes. This completes the proof.

\bibliographystyle{plain}
\bibliography{biblio}

@article {AG,
    AUTHOR = {Akbal, Y. and G\"ulo\u{g}lu, A. M.},
     TITLE = {Waring-{G}oldbach problem with {P}iatetski-{S}hapiro primes},
   JOURNAL = {J. Th\'eor. Nombres Bordeaux},
  FJOURNAL = {Journal de Th\'eorie des Nombres de Bordeaux},
    VOLUME = {30},
      YEAR = {2018},
    NUMBER = {2},
     PAGES = {449--467},
      ISSN = {1246-7405,2118-8572},
   MRCLASS = {11P32 (11L03 11P05 11P55)},
  MRNUMBER = {3891321},
MRREVIEWER = {Gang\ Yu},
       DOI = {10.5802/jtnb.1033},
       URL = {https://doi-org.ezproxy.uky.edu/10.5802/jtnb.1033},
}

@article {BF,
    AUTHOR = {Balog, A. and Friedlander, J.},
     TITLE = {A hybrid of theorems of {V}inogradov and
              {P}iatetski-{S}hapiro},
   JOURNAL = {Pacific J. Math.},
  FJOURNAL = {Pacific Journal of Mathematics},
    VOLUME = {156},
      YEAR = {1992},
    NUMBER = {1},
     PAGES = {45--62},
      ISSN = {0030-8730,1945-5844},
   MRCLASS = {11P32 (11N05)},
  MRNUMBER = {1182255},
MRREVIEWER = {Daniel\ A.\ Goldston},
       URL = {http://projecteuclid.org.ezproxy.uky.edu/euclid.pjm/1102635129},
}

@article{B17,
 author = {Bourgain, J.},
 title = {On the {Vinogradov} mean value},
 fjournal = {Proceedings of the Steklov Institute of Mathematics},
 journal = {Proc. Steklov Inst. Math.},
 issn = {0081-5438},
 volume = {296},
 pages = {30--40},
 year = {2017},
 language = {English; Russian},
 doi = {10.1134/S0081543817010035},
 zbMATH = {6740895},
 Zbl = {1371.11138}
}

@article {MSTT,
    AUTHOR = {Matom\"aki, K. and Shao, X. and Tao, T. and
              Ter\"av\"ainen, J.},
     TITLE = {Higher uniformity of arithmetic functions in short intervals
              {I}. {A}ll intervals},
   JOURNAL = {Forum Math. Pi},
  FJOURNAL = {Forum of Mathematics. Pi},
    VOLUME = {11},
      YEAR = {2023},
     PAGES = {Paper No. e29, 97},
      ISSN = {2050-5086},
   MRCLASS = {11N37 (11B30)},
  MRNUMBER = {4658200},
MRREVIEWER = {Peter\ Shiu},
       DOI = {10.1017/fmp.2023.28},
       URL = {https://doi-org.ezproxy.uky.edu/10.1017/fmp.2023.28},
}

@article {GT12,
    AUTHOR = {Green, B. and Tao, T.},
     TITLE = {The quantitative behaviour of polynomial orbits on
              nilmanifolds},
   JOURNAL = {Ann. of Math. (2)},
  FJOURNAL = {Annals of Mathematics. Second Series},
    VOLUME = {175},
      YEAR = {2012},
    NUMBER = {2},
     PAGES = {465--540},
      ISSN = {0003-486X,1939-8980},
   MRCLASS = {37A15},
  MRNUMBER = {2877065},
MRREVIEWER = {Tamar\ Ziegler},
       DOI = {10.4007/annals.2012.175.2.2},
       URL = {https://doi-org.ezproxy.uky.edu/10.4007/annals.2012.175.2.2},
}

@article {GT-linear,
    AUTHOR = {Green, B. and Tao, T.},
     TITLE = {Linear equations in primes},
   JOURNAL = {Ann. of Math. (2)},
  FJOURNAL = {Annals of Mathematics. Second Series},
    VOLUME = {171},
      YEAR = {2010},
    NUMBER = {3},
     PAGES = {1753--1850},
      ISSN = {0003-486X,1939-8980},
   MRCLASS = {11N13 (11B30 11P32)},
  MRNUMBER = {2680398},
MRREVIEWER = {Tamar\ Ziegler},
       DOI = {10.4007/annals.2010.171.1753},
       URL = {https://doi-org.ezproxy.uky.edu/10.4007/annals.2010.171.1753},
}

@article {GT-nil,
    AUTHOR = {Green, B. and Tao, T.},
     TITLE = {The {M}\"obius function is strongly orthogonal to
              nilsequences},
   JOURNAL = {Ann. of Math. (2)},
  FJOURNAL = {Annals of Mathematics. Second Series},
    VOLUME = {175},
      YEAR = {2012},
    NUMBER = {2},
     PAGES = {541--566},
      ISSN = {0003-486X,1939-8980},
   MRCLASS = {37A45 (11A25)},
  MRNUMBER = {2877066},
MRREVIEWER = {Tamar\ Ziegler},
       DOI = {10.4007/annals.2012.175.2.3},
       URL = {https://doi-org.ezproxy.uky.edu/10.4007/annals.2012.175.2.3},
}

@article {GT-kAP,
    AUTHOR = {Green, B. and Tao, T.},
     TITLE = {The primes contain arbitrarily long arithmetic progressions},
   JOURNAL = {Ann. of Math. (2)},
  FJOURNAL = {Annals of Mathematics. Second Series},
    VOLUME = {167},
      YEAR = {2008},
    NUMBER = {2},
     PAGES = {481--547},
      ISSN = {0003-486X,1939-8980},
   MRCLASS = {11N13 (11A41 11B25 37A45)},
  MRNUMBER = {2415379},
MRREVIEWER = {Tamar\ Ziegler},
       DOI = {10.4007/annals.2008.167.481},
       URL = {https://doi-org.ezproxy.uky.edu/10.4007/annals.2008.167.481},
}

@article {GTZ,
    AUTHOR = {Green, B. and Tao, T. and Ziegler, T.},
     TITLE = {An inverse theorem for the {G}owers {$U^{s+1}[N]$}-norm},
   JOURNAL = {Ann. of Math. (2)},
  FJOURNAL = {Annals of Mathematics. Second Series},
    VOLUME = {176},
      YEAR = {2012},
    NUMBER = {2},
     PAGES = {1231--1372},
      ISSN = {0003-486X,1939-8980},
   MRCLASS = {11B30},
  MRNUMBER = {2950773},
MRREVIEWER = {Julia\ Wolf},
       DOI = {10.4007/annals.2012.176.2.11},
       URL = {https://doi-org.ezproxy.uky.edu/10.4007/annals.2012.176.2.11},
}

@article {PS,
    AUTHOR = {Piatetski-Shapiro, I. I.},
     TITLE = {On the distribution of prime numbers in sequences of the form
              {$[f(n)]$}},
   JOURNAL = {Mat. Sbornik N.S.},
  FJOURNAL = {Mat. Sbornik N.S.},
    VOLUME = {33/75},
      YEAR = {1953},
     PAGES = {559--566},
   MRCLASS = {10.0X},
  MRNUMBER = {59302},
MRREVIEWER = {R.\ Bellman},
}

@article {Rivat-Wu,
    AUTHOR = {Rivat, J. and Wu, J.},
     TITLE = {Prime numbers of the form {$[n^c]$}},
   JOURNAL = {Glasg. Math. J.},
  FJOURNAL = {Glasgow Mathematical Journal},
    VOLUME = {43},
      YEAR = {2001},
    NUMBER = {2},
     PAGES = {237--254},
      ISSN = {0017-0895,1469-509X},
   MRCLASS = {11N05 (11N36)},
  MRNUMBER = {1838628},
MRREVIEWER = {Grigori\ Kolesnik},
       DOI = {10.1017/S0017089501020080},
       URL = {https://doi-org.ezproxy.uky.edu/10.1017/S0017089501020080},
}

@article {Rivat-Sargos,
    AUTHOR = {Rivat, J. and Sargos, P.},
     TITLE = {Nombres premiers de la forme {$\lfloor n^c\rfloor$}},
   JOURNAL = {Canad. J. Math.},
  FJOURNAL = {Canadian Journal of Mathematics. Journal Canadien de
              Math\'ematiques},
    VOLUME = {53},
      YEAR = {2001},
    NUMBER = {2},
     PAGES = {414--433},
      ISSN = {0008-414X,1496-4279},
   MRCLASS = {11N05 (11L07)},
  MRNUMBER = {1820915},
MRREVIEWER = {G.\ Greaves},
       DOI = {10.4153/CJM-2001-017-0},
       URL = {https://doi-org.ezproxy.uky.edu/10.4153/CJM-2001-017-0},
}

@article {Liu-Rivat,
    AUTHOR = {Liu, H. Q. and Rivat, J.},
     TITLE = {On the {P}iatetski-{S}hapiro prime number theorem},
   JOURNAL = {Bull. London Math. Soc.},
  FJOURNAL = {The Bulletin of the London Mathematical Society},
    VOLUME = {24},
      YEAR = {1992},
    NUMBER = {2},
     PAGES = {143--147},
      ISSN = {0024-6093,1469-2120},
   MRCLASS = {11N05 (11L07 11N36)},
  MRNUMBER = {1148674},
MRREVIEWER = {Grigori\ Kolesnik},
       DOI = {10.1112/blms/24.2.143},
       URL = {https://doi-org.ezproxy.uky.edu/10.1112/blms/24.2.143},
}

@article {Heath-Brown,
    AUTHOR = {Heath-Brown, D. R.},
     TITLE = {The {P}iatetski-{S}hapiro prime number theorem},
   JOURNAL = {J. Number Theory},
  FJOURNAL = {Journal of Number Theory},
    VOLUME = {16},
      YEAR = {1983},
    NUMBER = {2},
     PAGES = {242--266},
      ISSN = {0022-314X,1096-1658},
   MRCLASS = {10H20},
  MRNUMBER = {698168},
MRREVIEWER = {B.\ Garrison},
       DOI = {10.1016/0022-314X(83)90044-6},
       URL = {https://doi-org.ezproxy.uky.edu/10.1016/0022-314X(83)90044-6},
}

@article {Jia,
    AUTHOR = {Jia, C. H.},
     TITLE = {On {P}iatetski-{S}hapiro prime number theorem. {II}},
   JOURNAL = {Sci. China Ser. A},
  FJOURNAL = {Science in China (Scientia Sinica). Series A. Mathematics,
              Physics, Astronomy},
    VOLUME = {36},
      YEAR = {1993},
    NUMBER = {8},
     PAGES = {913--926},
      ISSN = {1001-6511},
   MRCLASS = {11N05 (11N36)},
  MRNUMBER = {1248537},
MRREVIEWER = {Grigori\ Kolesnik},
}

@article {Baker-Harman-Rivat,
    AUTHOR = {Baker, R. C. and Harman, G. and Rivat, J.},
     TITLE = {Primes of the form {$[n^c]$}},
   JOURNAL = {J. Number Theory},
  FJOURNAL = {Journal of Number Theory},
    VOLUME = {50},
      YEAR = {1995},
    NUMBER = {2},
     PAGES = {261--277},
      ISSN = {0022-314X,1096-1658},
   MRCLASS = {11N36 (11N05)},
  MRNUMBER = {1316821},
MRREVIEWER = {G.\ Greaves},
       DOI = {10.1006/jnth.1995.1020},
       URL = {https://doi-org.ezproxy.uky.edu/10.1006/jnth.1995.1020},
}

@article {Kumchev,
    AUTHOR = {Kumchev, A.},
     TITLE = {On the distribution of prime numbers of the form {$[n^c]$}},
   JOURNAL = {Glasg. Math. J.},
  FJOURNAL = {Glasgow Mathematical Journal},
    VOLUME = {41},
      YEAR = {1999},
    NUMBER = {1},
     PAGES = {85--102},
      ISSN = {0017-0895,1469-509X},
   MRCLASS = {11N36 (11N05)},
  MRNUMBER = {1689675},
MRREVIEWER = {R.\ C.\ Baker},
       DOI = {10.1017/S0017089599970477},
       URL = {https://doi-org.ezproxy.uky.edu/10.1017/S0017089599970477},
}

@article{SDP,
 author = {Sun, Y.-C. and Du, S.-S. and Pan, H.},
     TITLE = {Vinogradov's theorem with {P}iatetski-{S}hapiro primes},
   JOURNAL = {Int. Math. Res. Not. IMRN},
  FJOURNAL = {International Mathematics Research Notices. IMRN},
    VOLUME = {2025},
      YEAR = {2025},
    NUMBER = {15},
     PAGES = {Paper No. rnaf125, 30},
      ISSN = {1073-7928,1687-0247},
   MRCLASS = {11P32 (11B30)},
  MRNUMBER = {4939270},
MRREVIEWER = {Xianmeng\ Meng},
       DOI = {10.1093/imrn/rnaf125},
       URL = {https://doi.org/10.1093/imrn/rnaf125},
}

@article {MatthiesenWang,
    AUTHOR = {Matthiesen, L. and Wang, M.},
     TITLE = {Smooth numbers are orthogonal to nilsequences},
   JOURNAL = {Algebra Number Theory},
  FJOURNAL = {Algebra \& Number Theory},
    VOLUME = {19},
      YEAR = {2025},
    NUMBER = {10},
     PAGES = {1881--1946},
       DOI = {10.2140/ant.2025.19.1881},
       URL = {https://doi.org/10.2140/ant.2025.19.1881},
}

@article {Matthiesen2018,
    AUTHOR = {Matthiesen, L.},
     TITLE = {Generalized {Fourier} coefficients of multiplicative functions},
   JOURNAL = {Algebra Number Theory},
    VOLUME = {12},
      YEAR = {2018},
    NUMBER = {6},
     PAGES = {1311--1400},
       DOI = {10.2140/ant.2018.12.1311},
       URL = {https://doi.org/10.2140/ant.2018.12.1311},
}

@article {Robert2016,
    AUTHOR = {Robert, O.},
     TITLE = {On van der {Corput}'s {$k$}-th derivative test for exponential sums},
   JOURNAL = {Indag. Math. (N.S.)},
    VOLUME = {27},
      YEAR = {2016},
    NUMBER = {2},
     PAGES = {559--589},
       DOI = {10.1016/j.indag.2015.11.009},
       URL = {https://doi.org/10.1016/j.indag.2015.11.009},
}

@article {LiPan-gaps,
    AUTHOR = {Li, H. and Pan, H.},
     TITLE = {Small gaps between the {P}iatetski-{S}hapiro primes},
   JOURNAL = {Trans. Amer. Math. Soc.},
  FJOURNAL = {Transactions of the American Mathematical Society},
    VOLUME = {373},
      YEAR = {2020},
    NUMBER = {12},
     PAGES = {8463--8484},
       DOI = {10.1090/tran/8205},
       URL = {https://doi.org/10.1090/tran/8205},
}

@article {LiPan,
    AUTHOR = {Li, H. and Pan, H.},
     TITLE = {The {G}reen-{T}ao theorem for {P}iatetski-{S}hapiro primes},
   JOURNAL = {J. Funct. Anal.},
  FJOURNAL = {Journal of Functional Analysis},
    VOLUME = {291},
      YEAR = {2026},
    NUMBER = {1},
     PAGES = {Paper No. 111440, 25},
      ISSN = {0022-1236,1096-0783},
   MRCLASS = {11P32 (05D10 11B25 11L07 11N36)},
  MRNUMBER = {5050053},
       DOI = {10.1016/j.jfa.2026.111440},
       URL = {https://doi-org.ezproxy.uky.edu/10.1016/j.jfa.2026.111440},
}

@article {BST,
    AUTHOR = {Bienvenu, P.-Y. and Shao, X. and Ter\"av\"ainen,
              J.},
     TITLE = {A transference principle for systems of linear equations, and
              applications to almost twin primes},
   JOURNAL = {Algebra Number Theory},
  FJOURNAL = {Algebra \& Number Theory},
    VOLUME = {17},
      YEAR = {2023},
    NUMBER = {2},
     PAGES = {497--539},
      ISSN = {1937-0652,1944-7833},
   MRCLASS = {11B30 (11P32)},
  MRNUMBER = {4564765},
MRREVIEWER = {Ben\ Joseph\ Green},
       DOI = {10.2140/ant.2023.17.497},
       URL = {https://doi-org.ezproxy.uky.edu/10.2140/ant.2023.17.497},
}

@book{GKbook,
 author = {Graham, S. W. and Kolesnik, G.},
 title = {Van der {Corput}'s method for exponential sums},
 fseries = {London Mathematical Society Lecture Note Series},
 series = {Lond. Math. Soc. Lect. Note Ser.},
 issn = {0076-0552},
 volume = {126},
 isbn = {0-521-33927-8},
 year = {1991},
 publisher = {Cambridge etc.: Cambridge University Press},
 language = {English},
 zbMATH = {194767},
 Zbl = {0713.11001}
}

@article{CFZ,
 author = {Conlon, D. and Fox, J. and Zhao, Y.},
 title = {A relative {Szemer{\'e}di} theorem},
 fjournal = {Geometric and Functional Analysis. GAFA},
 journal = {Geom. Funct. Anal.},
 issn = {1016-443X},
 volume = {25},
 number = {3},
 pages = {733--762},
 year = {2015},
 language = {English},
 doi = {10.1007/s00039-015-0324-9},
 zbMATH = {6466325},
 Zbl = {1345.11008}
}

@article{RimanicWolf2019,
  author  = {Rimani{\'c}, L. and Wolf, J.},
  title   = {{Szemer{\'e}di}'s theorem in the primes},
  journal = {Proc. Edinb. Math. Soc. (2)},
  volume  = {62},
  number  = {2},
  pages   = {443--457},
  year    = {2019},
  doi     = {10.1017/S0013091518000561}
}

@misc{TeravainenWang2026,
  author        = {Ter{\"a}v{\"a}inen, J. and Wang, M.},
  title         = {On the {Green--Tao} theorem for sparse sets},
  year          = {2026},
  eprint        = {2603.09281},
  archivePrefix = {arXiv},
  primaryClass  = {math.NT},
  note          = {Preprint, arXiv:2603.09281},
  doi           = {10.48550/arXiv.2603.09281},
  url           = {https://arxiv.org/abs/2603.09281}
}

@article{TaoZiegler-poly,
  author  = {Tao, T. and Ziegler, T.},
  title   = {The primes contain arbitrarily long polynomial progressions},
  journal = {Acta Math.},
  volume  = {201},
  pages   = {213--305},
  year    = {2008},
  doi     = {10.1007/s11511-008-0032-5}
}

@article{CookMagyarTitichetrakun,
  author  = {Cook, B. and Magyar, {\'A}. and Titichetrakun, T.},
  title   = {A multidimensional {Szemer{\'e}di} theorem in the primes via combinatorics},
  journal = {Ann. Comb.},
  volume  = {22},
  pages   = {711--768},
  year    = {2018},
  doi     = {10.1007/s00026-018-0402-4}
}

@article{TaoZiegler-multidim,
  author  = {Tao, T. and Ziegler, T.},
  title   = {A multi-dimensional {Szemer{\'e}di} theorem for the primes via a correspondence principle},
  journal = {Israel J. Math.},
  volume  = {207},
  number  = {1},
  pages   = {203--228},
  year    = {2015},
  doi     = {10.1007/s11856-015-1157-9}
}

@inproceedings{KelleyMeka,
  author    = {Kelley, Z. and Meka, R.},
  title     = {Strong bounds for {$3$}-progressions},
  booktitle = {2023 IEEE 64th Annual Symposium on Foundations of Computer Science (FOCS)},
  publisher = {IEEE},
  pages     = {933--973},
  year      = {2023},
  doi       = {10.1109/FOCS57990.2023.00059},
  url       = {https://doi.org/10.1109/FOCS57990.2023.00059}
}

@article{Mirek2015,
  author  = {Mirek, M.},
  title   = {{Roth}'s theorem in the {Piatetski-Shapiro} primes},
  journal = {Rev. Mat. Iberoam.},
  volume  = {31},
  number  = {2},
  pages   = {617--656},
  year    = {2015},
  doi     = {10.4171/RMI/848},
  url     = {https://ems.press/journals/rmi/articles/13283}
}

@article{LuGuoGuo2026,
  author  = {Lu, Li and Guo, Lingyu and Guo, Victor Zhenyu},
  title   = {Improvements on exponential sums related to
             {Piatetski-Shapiro} primes},
  journal = {J. Number Theory},
  volume  = {281},
  year    = {2026},
  pages   = {700--725},
  doi     = {10.1016/j.jnt.2025.10.015}
}

\end{document}